\documentclass[a4paper,11pt,reqno]{amsart}  
\usepackage[a4paper,hmargin={2.5cm,2.5cm},vmargin={2cm,2cm}]{geometry}

\usepackage[square]{natbib} %,numbers
\usepackage{amsmath}
\usepackage{amssymb,amsthm}
\usepackage{bbm}
\usepackage{dsfont}
\usepackage[all]{xy}

\theoremstyle{plain}
\newtheorem{theorem}{Theorem}[section]
\newtheorem{lemma}[theorem]{Lemma}
\newtheorem{proposition}[theorem]{Proposition}
\newtheorem{corollary}[theorem]{Corollary}

\theoremstyle{definition}
\newtheorem{definition}[theorem]{Definition}
\newtheorem{examples}[theorem]{Examples}
\newtheorem{example}[theorem]{Example}

\theoremstyle{remark}
\newtheorem{remark}[theorem]{Remark}

\newtheorem{assumption}{Assumption}

\newenvironment{eqcond}{\begin{enumerate}}{\end{enumerate}}

\newcommand{\Rw}{\Rightarrow}

\newcommand{\hrw}{\hookrightarrow}

\newcommand{\RLw}{\Leftrightarrow}

\newcommand{\ff}{\mathfrak{f}}

\newcommand{\fx}{\mathfrak{x}}

\newcommand{\calA}{\mathcal{A}}

\newcommand{\fF}{\mathfrak{F}}

\DeclareMathOperator{\Fix}{Fix}

\DeclareMathOperator{\downc}{\downarrow\!}
\DeclareMathOperator{\ev}{ev}

\newcommand{\mate}[1]{\,^\ulcorner\! #1^\urcorner}

\newcommand{\catfont}[1]{\mathsf{#1}}

\newcommand{\SET}{\catfont{Set}}
\newcommand{\REL}{\catfont{Rel}}
\newcommand{\ORD}{\catfont{Ord}}
\newcommand{\SFRM}{\catfont{SFrm}}
\newcommand{\TOP}{\catfont{Top}}
\newcommand{\TOPREL}{\catfont{TopRel}}
\newcommand{\ORDCH}{\catfont{OrdCompHaus}}
\newcommand{\STCOMP}{\catfont{StablyComp}}
\newcommand{\SPECREL}{\catfont{SpecRel}}
\newcommand{\SPEC}{\catfont{Spec}}
\newcommand{\STONE}{\catfont{Stone}}
\newcommand{\STONEREL}{\catfont{StoneRel}}
\newcommand{\COMPHAUS}{\catfont{CompHaus}}
\newcommand{\SOB}{\catfont{Sob}}
\newcommand{\CALG}{C^*\text{-}\catfont{Alg}}
\newcommand{\DLAT}{\catfont{DLat}}
\newcommand{\BOOL}{\catfont{Bool}}
\newcommand{\CABOOL}{\catfont{CABool}}

\newcommand{\Coalg}[1]{\catfont{Coalg}(#1)}
\newcommand{\CatOp}[2]{\catfont{#1}\text{-}\catfont{Op}(#2)}

\newcommand{\Mon}{\catfont{Mon}}
\newcommand{\LAdj}{\catfont{LAdj}}
\newcommand{\RAdj}{\catfont{RAdj}}

\newcommand{\EMX}{\catfont{X}^{(\_)}}
\newcommand{\KlX}{\catfont{X}_{(\_)}}

\newcommand{\Card}{\catfont{Card}}

\newcommand{\Alg}[1]{#1\text{-}\catfont{Alg}}

\newcommand{\relto}{{\longrightarrow\hspace*{-2.8ex}{\mapstochar}\hspace*{2.6ex}}}

\newcommand{\kto}{\relbar\joinrel\rightharpoonup}

\newcommand{\monadfont}[1]{\mathbbm{#1}}

\newcommand{\mT}{\monadfont{T}}
\newcommand{\mI}{\monadfont{1}}

\newcommand{\mV}{\monadfont{V}}
\newcommand{\mtV}{\hat{\monadfont{V}}}
\newcommand{\mF}{\monadfont{F}}
\newcommand{\mP}{\monadfont{P}}

\newcommand{\monad}{(T,e,m)}
\newcommand{\imonad}{(1,1,1)}

\newcommand{\fmonad}{(F,e,m)}
\newcommand{\ftmonad}{(F_\tau,e,m)}
\newcommand{\vmonad}{(V,e,m)}
\newcommand{\pmonad}{(P,e,m)}
\newcommand{\tvmonad}{(\hat{V},e,m)}

\newcommand\adjunct[2]{\xymatrix@=8ex{\ar@{}[r]|{\top}\ar@<1mm>@/^2mm/[r]^{{#2}} & \ar@<1mm>@/^2mm/[l]^{{#1}}}}

\newcommand{\adjunction}{(F\dashv G,\eta,\eps):\catfont{A}\rightleftarrows\catfont{X}}
\newcommand{\adjunctionb}{(F'\dashv G',\eta',\eps'):\catfont{A}'\rightleftarrows\catfont{X}}
\newcommand{\adjunctionT}{(F^\mT\dashv G^\mT,\eta,\eps):\catfont{X}^\mT\rightleftarrows\catfont{X}}

\newcommand{\doo}[1]{\overset{\centerdot}{#1}}
\newcommand{\eps}{\varepsilon}
\newcommand{\op}{\mathrm{op}}
\newcommand{\can}{\mathrm{can}}
\newcommand{\sep}{\mathrm{sep}}

\newcommand{\df}[1]{\emph{\textbf{#1}}}

\title{Dualities in modal logic from the point of view of triples}
\author{Dirk Hofmann}
\author{Pedro Nora}
\thanks{Partial financial assistance by FEDER funds through COMPETE -- Operational Programme Factors of Competitiveness (Programa Operacional Factores de Competitividade) and by Portuguese funds through the Center for Research and Development in Mathematics and Applications (University of Aveiro) and the Portuguese Foundation for Science and Technology (FCT -- Funda\c{c}\~ao para a Ci\^encia e a Tecnologia), within project PEst-C/MAT/UI4106/2011 with COMPETE number FCOMP-01-0124-FEDER-022690, and the project MONDRIAN under the contract PTDC/EIA-CCO/108302/2008 with COMPETE number FCOMP-01-0124-FEDER-010047 is gratefully acknowledge.}
\address{Center for Research and Development in Mathematics and Applications, Department of Mathematics, University of Aveiro, 3810-193 Aveiro, Portugal}
\email{dirk@ua.pt}
\email{a28224@ua.pt}
\date{\today}


\usepackage[hypertexnames=false]{hyperref} 

\begin{document}

\begin{abstract}
In this paper we show how the theory of monads can be used to deduce in a uniform manner several duality theorems involving categories of relations on one side and categories of algebras with homomorphisms preserving only some operations on the other. Furthermore, we investigate the monoidal structure induced by Cartesian product on the relational side and show that in some cases the corresponding operation on the algebraic side represents bimorphisms.
\end{abstract}

\maketitle

\section*{Introduction}

The title (and content) of this note is clearly inspired by the paper \citep{Neg71} where the author derives the classical duality theorems of Gelfand and Pontrjagin ``employing the theory of triples and using only a minimum of analytic information'' \citep{Neg71}. We recall that Gelfand's duality theorem states that the category $\CALG$ of commutative and unital $C^*$-algebras and homomorphisms is equivalent to the dual of the category $\COMPHAUS$ of compact Hausdorff spaces and continuous maps. The argument of \citep{Neg71} can  be summarised as follows:
\begin{itemize}
\item the categories $\CALG$ and $\COMPHAUS^\op$ are both tripleable ($=$ monadic) over $\SET$, that is, $\COMPHAUS^\op$ is equivalent to the category of Eilenberg--Moore algebras of some triple ($=$ monad), and also $\CALG$ is equivalent to the category of Eilenberg--Moore algebras of some monad (the proof of the latter fact needs the ``minimum of analytic information'');
\item the induced monads are equal.
\end{itemize}

At the heart of the equivalence between Kripke semantics and algebraic semantics in modal logic lies J\'onnson and Tarski's representation theorem for Boolean algebras with operator (see \cite{JT51,JT52}), which eventually led to a duality between the category of Boolean algebras with operator and homomorphisms and the category of ``Stone Kripke frames''. The latter category is elegantly described in \citep{KKV04} as the category of coalgebras for the Vietoris functor on the category $\STONE$ of Stone spaces and continuous maps. As advocated in \citep{Hal56} (see also \citep{Wri57}) and later in \citep{SV88}, this duality result can be seen as a consequence of the more general duality between the category $\STONEREL$ of Stone spaces and Boolean relations and the category of Boolean algebras with ``hemimorphisms'', that is, maps preserving finite suprema but not necessarily finite infima.

The starting point of this paper is the observation that the Vietoris functor is actually part of a monad on $\STONE$, and $\STONEREL$ is equivalent to the Kleisli category of this monad. This fact opens the door to use an argumentation similar to the one in \citep{Neg71}, but now with the Kleisli construction \emph{in lieu} of the Eilenberg--Moore construction. We hasten to remark that the situation here is actually much simpler than the one described above, basically because it is usually easy the see that a category is equivalent to a Kleisli category whereby monadicity is a property much harder to establish. Therefore, using the theory of monads, in this paper we derive in a uniform way several duality theorems involving categories of relations  and categories of algebras with ``hemimorphisms''. Our examples include Halmos duality as well as a similar result for the category $\SPECREL$ of spectral spaces and spectral relations (respectively Priestley spaces and Priestley relations, see \citep{CLP91}). We then proceed by investigating the monoidal structure on $\SPECREL$ and $\STONEREL$ induced by the topological product of spaces, and show that this structure corresponds under the aforementioned dualities to a tensor product which represents bimorphisms. 

\section{Kleisli adjunctions}

The purpose of this section is to recall briefly some well-known facts about monads and adjunctions. The main focus lies on the fact that the principal constructions are functorial and give rise to (large) adjunctions as explained in  \citep{Pum70,Pum88} and \citep{Tho74} (see also \citep{Por94}). For more information on monads we refer to \citep{MS04}.

\begin{definition}
An \df{adjunction} $\adjunction$ consists of
\begin{itemize}
\item the right adjoint functor $G:\catfont{A}\to\catfont{X}$,
\item the left adjoint functor $F:\catfont{X}\to\catfont{A}$,
\item the unit natural transformation $\eta:1_\catfont{X}\to GF$, and
\item the co-unit natural transformation $\eps:FG\to 1_\catfont{A}$
\end{itemize}
such that the diagrams
\begin{align*}
\xymatrix{F\ar[r]^-{F\eta}\ar@{=}[dr]_1 & FGF\ar[d]^-{\eps_F}\\ & F}
&  &
\xymatrix{G\ar[r]^-{\eta_G}\ar@{=}[dr]_1 & GFG\ar[d]^-{G\eps}\\ & G}
\end{align*}
commute. Let $\adjunction$ and $\adjunctionb$ be adjunctions over a fixed base category $\catfont{X}$. A \df{right morphism of adjunctions} over $\catfont{X}$ is a functor $J:\catfont{A}\to\catfont{A}'$ with $G=G'J$. Likewise, a \df{left morphism of adjunctions} over $\catfont{X}$ is a functor $J:\catfont{A}\to\catfont{A}'$ with $F'=JF$.
\end{definition}

We denote by $\RAdj(\catfont{X})$ the category of adjunctions over $\catfont{X}$ and right morphisms of adjunctions over $\catfont{X}$, and $\LAdj(\catfont{X})$ denotes the category of adjunctions over $\catfont{X}$ and left morphisms of adjunctions over $\catfont{X}$. 

\begin{remark}
Note that we do \emph{not} require $F'=JF$ in the definition of a right morphism of adjunctions; however, there is a canonical natural transformation $\kappa:F'\to JF$ defined as the composite
\[
F'\xrightarrow{F'\eta} F'GF=F'G'JF\xrightarrow{\eps_{JF}}JF,
\]
and for a left morphism of adjunctions $J$ we have a canonical natural transformation $\iota:G\to G'J$ defined as the composite
\[
G\xrightarrow{\eta'_G}G'F'G=G'JFG\xrightarrow{G'J\eps}G'J.
\]
\end{remark}

\begin{remark}\label{rem:FixedSubcat}
An adjunction is an \df{equivalence} whenever both the unit and the co-unit are natural isomorphisms. Every adjunction $\adjunction$ induces an equivalence between the (possibly empty) full subcategories
\begin{align*}
\Fix(\catfont{X})=\{X\in\catfont{X}\mid \eta_X\text{ is an isomorphism}\} &&\text{and}&&
\Fix(\catfont{A})=\{A\in\catfont{A}\mid \eps_A\text{ is an isomorphism}\}.
\end{align*}
\end{remark}

\begin{definition}
A \df{monad} $\mT=\monad$ on a category $\catfont{X}$ consists of a functor $T:\catfont{X}\to\catfont{X}$ together with natural transformations $e:1_\catfont{X}\to T$ (unit) and $m:TT\to T$ (multiplication) such that the diagrams
\begin{align*}
\xymatrix{T^3\ar[r]^{m_T}\ar[d]_{Tm} & T^2\ar[d]^m\\ T^2\ar[r]_m & T}
&&  &&
\xymatrix{T\ar[r]^{e_T}\ar[rd]_{1_T} & T^2\ar[d]^m & T\ar[l]_{Te}\ar[ld]^{1_T}\\ & T}
\end{align*}
commute. For monads $\mT,\mT'$ on a category $\catfont{X}$, a \df{monad morphism} $j:\mT\to\mT'$ is a natural transformation $j:T\to T'$ such that the diagrams
\begin{align*}
\xymatrix{ &1\ar[dl]_e\ar[dr]^{e'}\\ T\ar[rr]_j & & T'}
&&  &&
\xymatrix{TT\ar[r]^{j^2}\ar[d]_m & T'T'\ar[d]^{m'}\\ T\ar[r]_j & T'}
\end{align*}
commute, where $j^2=j_{T'}\cdot Tj=T'j\cdot j_T$.
\end{definition}
The category of monads on $\catfont{X}$ and monad morphisms is denoted by $\Mon(\catfont{X})$.

\begin{examples}
 \begin{enumerate}
\item The \df{identity monad} $\mI=\imonad$. Trivially, the identity functor $1:\catfont{X}\to\catfont{X}$ together with the identity transformation $1:1\to 1$ forms a monad. It is the initial monad: for every monad $\mT=\monad$ on $\catfont{X}$, the unit $e$ is the unique monad morphism $\mI\to\mT$.
\item The \df{powerset monad} $\mP=\pmonad$ on $\SET$. The powerset functor $P:\SET\to\SET$ sends each set to its powerset $PX$ and each function $f:X\to Y$ to the direct image function $Pf:PX\to PY,A\mapsto f[A]$. The $X$-component of the natural transformation $e$ respectively $m$ is given by ``taking singletons'' $e_X:X\to PX,x\mapsto \{x\}$ and union $m_X:PPX\to PX,\calA\mapsto\bigcup\calA$.
\item The \df{filter monad} $\mF=\fmonad$ on $\SET$. The filter functor $F:\SET\to\SET$ sends a set $X$ to the set $FX$ of all filters on $X$ and, for $f:X\to Y$, the map $Ff$ sends a filter $\ff$ on $X$ to the filter $\{B\subseteq Y\mid f^{-1}[B]\in\ff\}$ on $Y$. The natural transformations $e:1\to F$ and $m:FF\to F$ are given by 
\begin{align*}
&&e_X(x)&=\doo{x}=\{A\subseteq X\mid x\in A\}&\text{and}&&m_X(\fF)&=\{A\subseteq X\mid A^\#\in\fF\}, 
\end{align*}
for all sets $X$, $\fF\in FFX$ and $x\in X$, where $A^\#=\{\ff\in FX\mid A\in\ff\}$.
\item The \df{filter monad} $\mF=\ftmonad$ on $\TOP$. For a topological space $X$, $F_\tau X$ is the set of all filters on the lattices of opens of $X$, equipped with the topology generated by the sets $U^\#$, for $U\subseteq X$ open. The continuous map $F_\tau f:F_\tau X\to F_\tau Y$, for $f:X\to Y$ in $\TOP$, and the unit and the multiplication are defined as above. For more information we refer to \citep{Esc97}.
\end{enumerate}
In Section \ref{sect:StablyComp} we will describe topological counterparts of the powerset monad.
\end{examples}

In the remainder of this section we will construct adjunctions
\begin{align*}
 \Mon(\catfont{X})^\op\rightleftarrows\RAdj(\catfont{X}) &&\text{and}&& \LAdj(\catfont{X})\rightleftarrows\Mon(\catfont{X})
\end{align*}
and identify their fixed objects. We explain how these results can be used to establish equivalences of categories.

Every adjunction $\adjunction$ over $\catfont{X}$ induces a monad $\mT=\monad$ on $\catfont{X}$ defined by
\begin{align*}
T=GF,&& e=\eta && \text{and} && m=G\eps_F.
\end{align*}
Moreover, every right morphism $J$ of adjunctions $\adjunction$ and $\adjunctionb$ over $\catfont{X}$ induces a monad morphism
\[j=G'\kappa:\mT'\to\mT\]
between the induced monads. These constructions define the object and the morphism part of the functor $M^\catfont{X}:\RAdj(\catfont{X})\to\Mon(\catfont{X})^\op$. Likewise, every left morphism $J$ of adjunctions induces a monad morphism
\[j=\iota_F:\mT\to\mT'\]
and we obtain a functor $M_\catfont{X}:\LAdj(\catfont{X})\to\Mon(\catfont{X})$. As we show next, both functors have adjoints given by well-known constructions.

Let $\mT=\monad$ be a monad on a category $\catfont{X}$. A \df{$\mT$-algebra} (also called \df{Eilenberg-Moore algebra}) is a pair $(X,\alpha)$ consisting of an $\catfont{X}$-object $X$ and an $\catfont{X}$-morphism $\alpha:TX\to X$ making the diagrams
\begin{align*}
\xymatrix{X\ar[r]^{e_X}\ar@{=}[dr]_{1_X} & TX\ar[d]^\alpha\\ & X}
&&  &&
\xymatrix{TTX\ar[r]^{m_X}\ar[d]_{T\alpha} & TX\ar[d]^\alpha\\ TX\ar[r]_\alpha & X}
\end{align*}
commutative. Let $(X,\alpha)$ and $(Y,\beta)$ be $\mT$-algebras. A map $f:X\to Y$ is a \df{$\mT$-algebra homomorphism} if the diagram
\begin{align*}
\xymatrix{TX\ar[d]_\alpha\ar[r]^{Tf} & TY\ar[d]^\beta\\ X\ar[r]_f & Y}
\end{align*}
commutes. The category of $\mT$-algebras and $\mT$-algebra homomorphisms is denoted by $\catfont{X}^\mT$. There is a canonical forgetful functor $G^\mT:\catfont{X}^\mT\to\catfont{X},(X,\alpha)\mapsto X$ with left adjoint $F^\mT:\catfont{X}\to\catfont{X}^\mT,X\mapsto (TX,m_X)$. Moreover, every monad morphism $j:\mT\to\mT'$ induces a functor
\[
\catfont{X}^j:\catfont{X}^{\mT'}\to\catfont{X}^\mT, (X,\alpha)\mapsto(X,\alpha\cdot j_X)
\]
with $G^\mT \catfont{X}^j=G^{\mT'}$, that is, $\catfont{X}^j:(F^{\mT'}\dashv G^{\mT'})\to(F^\mT\dashv G^\mT)$ is a right morphism of adjunctions over $\catfont{X}$. These constructions define indeed a functor
\[\EMX:\Mon(\catfont{X})^\op\to\RAdj(\catfont{X}).\]
It is easy to see that $F^\mT\dashv G^\mT$ induces $\mT$, that is, $\mT=M^\catfont{X}\cdot\EMX(\mT)$. On the other hand, for every adjunction $\adjunction$ over $\catfont{X}$ (with induced monad $\mT$) we have a canonical \df{comparison functor} $K:\catfont{A}\to\catfont{X}^\mT$ defined by $K(A)=(GA,G\eps_A)$ and $Kf=Gf$; hence,
\[K:(\adjunction)\to(\adjunctionT)\]
is a right morphism of adjunctions over $\catfont{X}$. For a right morphism $J$ of adjunctions $\adjunction$ and $\adjunctionb$ over $\catfont{X}$, the diagram
\[
  \xymatrix{\catfont{A}\ar[r]\ar[d]_J & \catfont{X}^\mT\ar[d]^{\catfont{X}^j} \\ \catfont{A}'\ar[r] & \catfont{X}^{\mT'}}
\]
commutes, that is, the family of all comparison functors defines a natural transformation $1\to\EMX M^\catfont{X}$. In fact, this transformation together with the family $(\mT=M^\catfont{X}\EMX(\mT))_\mT$ are the units of the adjunction
\begin{equation*}
M^\catfont{X}\dashv \EMX:\Mon(\catfont{X})^\op\rightleftarrows{\RAdj(\catfont{X})}.
\end{equation*}

Clearly, $\Fix(\Mon(\catfont{X}))=\Mon(\catfont{X})$; but, deviating slightly from Remark \ref{rem:FixedSubcat}, we let the \df{fixed} subcategory $\Fix(\RAdj(\catfont{X}))$ consist of those objects whose component of the unit is an equivalence of categories. An object of $\Fix(\RAdj(\catfont{X}))$ is called \df{monadic adjunction}, these adjunctions are characterised by the following

\begin{theorem}[\citep{Bec67}]
An adjunction $\adjunction$ is monadic if and only if $G$ reflects isomorphisms and $\catfont{A}$ has and $G$ preserves all $G$-contractible coequaliser pairs. 
\end{theorem}

\begin{example}\label{ex:MonadicCompHaus}
The canonical forgetful functor $|-|:\COMPHAUS\to\SET$ from the category of compact Hausdorff spaces and continuous maps has a left adjoint given by \v{C}ech--Stone compactification, and it is shown in \citep{Man69} that this adjunction is monadic. The induced monad on $\SET$ is the ultrafilter monad.
\end{example}

The ``equivalence'' between monads and monadic adjunction provides a general principle to prove equivalence of two categories: firstly, show that both categories are part of a monadic adjunction over the same category $\catfont{X}$; and secondly, show that this adjunctions induce isomorphic monads. This idea was used in \citep{Neg71} to obtain the classical duality theorems of Gelfand and Pontrjagin.

We will explain now how $M_\catfont{X}:\LAdj(\catfont{X})\to\Mon(\catfont{X})$ is part of an adjunction. Let $\mT=\monad$ be a monad over $\catfont{X}$. The \df{Kleisli category} $\catfont{X}_\mT$ has the same objects as $\catfont{X}$, and a morphism $f:X\kto Y$ in $\catfont{X}_\mT$ is an $\catfont{X}$-morphism $f:X\to TY$. Given morphisms $f:X\kto Y$ and $g:Y\kto Z$ in $\catfont{X}_\mT$, the composite $g\cdot f:X\kto Z$ is defined as
\[
X\xrightarrow{\;f\;}TY\xrightarrow{Tg}TTZ\xrightarrow{m_Z}TZ.
\]
Then $e_X:X\to TX$ is the identity on $X$ in $\catfont{X}_\mT$. We have a canonical adjunction $F_\mT\dashv G_\mT:\catfont{X}_\mT\rightleftarrows\catfont{X}$, where
\begin{align*}
G_\mT:\catfont{X}_\mT \to \catfont{X},\;\;&
f:X\kto Y \mapsto TX\xrightarrow{Tf}TTY\xrightarrow{m_Y}TY\\
\intertext{and}
F_\mT:\catfont{X} \to \catfont{X}_\mT,\;\;&
f:X\to Y \mapsto X\xrightarrow{f}Y\xrightarrow{e_Y} TY.
\end{align*}

\begin{example}
For the powerset monad $\mP$ on $\SET$, the category $\SET_\mP$ is equivalent to the category $\REL$ of sets and relations by interpreting a map $f:X\to PY$ as a relation $X\relto Y$ from $X$ to $Y$. Under this equivalence, $F_\mP:\SET\to\SET_\mP$ corresponds to the inclusion functor $\SET\to\REL$ and $G_\mP:\SET_\mP\to\SET$ corresponds to the functor $\REL\to\SET$ which sends a set $X$ to its powerset $PX$ and a relation $r:X\relto Y$ to the map $PX\to PY,\,A\mapsto r[A]$.
\end{example}

Every monad morphism $j:\mT\to\mT'$ induces a functor $\catfont{X}_j:\catfont{X}_\mT\to\catfont{X}_{\mT'}$ which acts as the identity on objects and sends $f:X\kto Y$ to $X\xrightarrow{f}TY\xrightarrow{j_Y}T'Y$. One clearly has $F_{\mT'}=\catfont{X}_j F_\mT$, hence $\catfont{X}_j$ is a left morphism of adjunctions and we obtain a functor 
\[\KlX:\Mon(\catfont{X})\to\LAdj(\catfont{X}).\] 
As before, the induced monad of $F_\mT\dashv G_\mT$ is $\mT$ again, that is, $\mT=M_\catfont{X}\cdot\KlX(\mT)$. For every adjunction $\adjunction$ over $\catfont{X}$ (with induced monad $\mT$) we have a canonical \df{comparison functor} $C:\catfont{X}_\mT\to\catfont{A}$ sending an object $X$ in $\catfont{X}_\mT$ to $FX$ and a morphism $f:X\kto Y$ to $FX\xrightarrow{Ff}FTY=FGFY\xrightarrow{\eps_{FY}} FY$. Since $F=C F_\mT$, $C$ is a left morphism of adjunctions. For a left morphism $J$ of adjunctions $\adjunction$ and $\adjunctionb$ over $\catfont{X}$, the diagram
\begin{equation}\label{diag:NatUpEq2}
  \xymatrix{\catfont{X}_\mT\ar[r]\ar[d]_{\catfont{X}j} & \catfont{A}\ar[d]^J \\ \catfont{X}_{\mT'}\ar[r] & \catfont{A}' }
\end{equation}
commutes. Therefore $C$ is the ($\adjunction$)-component of a natural transformation $\KlX M_\catfont{X}\to 1$, in fact, this transformation is the co-unit of the adjunction
\[
\KlX\dashv M_\catfont{X}:\LAdj(\catfont{X})\rightleftarrows\Mon(\catfont{X}),
\]
where the unit $1\to M_\catfont{X}\KlX$ is given by $(\mT=M_\catfont{X}\KlX(\mT))_\mT$. The comparison functor $C:\catfont{X}_\mT\to\catfont{A}$ is always fully faithful, and we call an adjunction $\adjunction$ a \df{Kleisli adjunction} whenever $C$ is an equivalence. Unlike the situation for monadic adjunctions, Kleisli adjunctions can be easily characterised.

\begin{theorem}
An adjunction $\adjunction$ is a Kleisli adjunction if and only if $F$ is essentially surjective on objects.
\end{theorem}

As for monadic adjuntions, \eqref{diag:NatUpEq2} gives a simple scheme to obtain an equivalence between categories $\catfont{A}$ and $\catfont{A}'$:
\begin{theorem}\label{thm:KleisiEq}
A functor $J:\catfont{A}\to\catfont{A}'$ between Kleisli adjunctions $\adjunction$ and $\adjunctionb$ is an equivalence provided that $F'=JF$ and the morphism $M_\catfont{X}(J)$ between the induced monads is a natural isomorphism.
\end{theorem}

We will illustrate this principle with various examples in Section \ref{sect:KleisliDual}.

\section{Stably compact spaces, spectral spaces, and Stone spaces}\label{sect:StablyComp}

Our main examples in the following sections involve Stone spaces and spectral spaces as well as the (lower) Vietoris monad. We recall that these spaces where introduced by M.H.\ Stone in \citep{Sto36,Sto38} in order to give representation theorems for Boolean algebras and distributive lattices. Both are examples of stably compact topological spaces, and stably compact spaces can be equivalently described as compact Hausdorff spaces equipped with a compatible order relation. We stress that, in order to describe their properties, it is quite useful to pass freely from one description to the other. In this section we collect a couple of (categorical) properties of stably compact spaces, we will be particularly interested in a description of initial structures. For more information on these types of spaces we refer to \citep{Jun04}, \citep{Law11} and \citep{Tho09}, and on topological functors to \citep{AHS90}. There exists a vast literature on the (lower) Vietoris construction, see \citep{Vie22}, \citep {Pop66}, \citep{CT97} and \citep{KKV04}, for instance. A study of stably compact spaces and the Vietoris monad in a more general context can be found in \citep{Hof12a}.

For a topological space $X$, we consider its \df{underlying order} defined by $x\le y$ whenever $y\in\overline{\{x\}}$, or, equivalently, whenever the principal filter $\doo{x}$ converges to $y$. Clearly, every continuous map is monotone with respect to this order. We also note that this order relation is dual to the specialisation order. For continuous maps $f:X\to Y$ and $g:Y\to X$, we say that $f$ is \df{left adjoint} to $g$, written as $f\dashv g$, if $x\le g(f(x))$ and $f(g(y))\le y$, for all $x\in X$ and $y\in Y$. Given $f$, there exists up to equivalence at most one such $g$, and in this case we call $f$ a left adjoint continuous map.

An \df{ordered compact Hausdorff space} (see \citep{Nac50}) is a triple $(X,\le,\tau)$ where $(X,\le)$ is an ordered set and $\tau$ is a compact Hausdorff topology on $X$ so that $\{(x,y)\mid x\le y\}$ is closed in $X\times X$. We follow here \citep{Tho09} and do not assume the order relation $\le$ on $X$ to be anti-symmetric. In the sequel we write $\ORDCH$ for the category of ordered compact Hausdorff spaces and maps preserving both the order and the topology, and denote its full subcategory defined by the objects with anti-symmetric order as $\ORDCH_\sep$. As shown in \citep{Tho09}, $\ORDCH$ is the Eilenberg--Moore category for the ultrafilter monad on the category $\ORD$ of (not necessarily anti-symmetric) ordered sets and monotone maps. Therefore $\ORDCH$ is complete and the forgetful functor $|-|:\ORDCH\to\ORD$ preserves limits. Given an ordered compact Hausdorff space  $(X,\le,\tau)$, one defines a new topology $\downc\tau$ on $X$ which contains precisely those elements of $\tau$ which are down-closed, and this construction defines a functor $K:\ORDCH\to\TOP$. Note that the underlying order of $KX$ is precisely the order relation of $X$. We obtain a commuting diagram of functors
\[
 \xymatrix{\ORDCH\ar[d]_{|-|}\ar[r]^-{K} & \TOP\ar[d]^{|-|}\\ \COMPHAUS\ar[r]_-{|-|} & \SET}
\]
where both vertical arrows represent topological functors (see \citep{Tho09}) and both horizontal arrows monadic functors (see Example \ref{ex:MonadicCompHaus} and \citep{Sim82}, the latter only considers anti-symmetric orders; however, the general case is similar as pointed out in \citep{Hof13}). Hence, $\ORDCH$ is also cocomplete and the topological functor $|-|:\ORDCH\to\COMPHAUS$ has a fully faithful right adjoint and a fully faithful left adjoint $\COMPHAUS\hrw\ORDCH$; the latter equips a compact Hausdorff space with the discrete order. This functor clearly corestricts to $\COMPHAUS\hrw\ORDCH_\sep$ and this corestriction is left adjoint to the forgetful functor $|-|:\ORDCH_\sep\to\COMPHAUS$.

\begin{proposition}\label{prop:InitialSources}
Let $(f_i:X\to X_i)_{i\in I}$ be a source in $\ORDCH$. Then $(f_i:X\to X_i)_{i\in I}$ is initial with respect to $|-|:\ORDCH\to\COMPHAUS$ if and only if $(Kf_i:KX\to KX_i)_{i\in I}$ is initial with respect to $|-|:\TOP\to\SET$.
\end{proposition}
\begin{proof}
This follows from the description of the initial lifts of \citep[Proposition 3]{Tho09}.
\end{proof}

It is also shown in \citep{Tho09} that $\ORDCH_\sep$ is a quotient-reflective subcategory of $\ORDCH$, and therefore:

\begin{proposition}
$\ORDCH_\sep\hrw\ORDCH$ is closed under initial mono-sources with respect to $|-|:\ORDCH\to\COMPHAUS$. Hence, $\ORDCH_\sep$ is complete and cocomplete with limits formed as in $\ORDCH$. The forgetful functor $|-|:\ORDCH_\sep\to\COMPHAUS$ is mono-topological and therefore every source in $\ORDCH_\sep$ factors as an epimorphism followed by an initial mono-source (with respect to $|-|:\ORDCH_\sep\to\COMPHAUS$). 
\end{proposition}

In all categories considered above, a source is a mono-source precisely if it is point-separating.

A topological space $X$ is called \df{stably compact} whenever $X$ is a locally compact, sober, and the compact down-closed subsets are closed under finite intersections. A continuous map $f:X\to Y$ between stably compact spaces is \df{spectral map} whenever $f^{-1}(K)$ is compact for every compact down-closed subset $K\subseteq Y$. A useful criterion is given by the fact that every left adjoint continuous map between stably compact spaces is spectral. The category of stably compact spaces and spectral maps is denoted as $\STCOMP$.

\begin{theorem}\label{thm:OCHvsSTC}
The functor $K:\ORDCH\to\TOP$ restricts to an isomorphism \[K:\ORDCH_\sep\to\STCOMP.\] Hence, $\STCOMP$ is complete and cocomplete and the inclusion functor $\STCOMP\to\TOP$ preserves limits.
\end{theorem}
\begin{proof}
See \citep{Jun04}, for instance.
\end{proof}

Via this equivalence, $|-|:\ORDCH_\sep\to\COMPHAUS$ corresponds to a mono-topological faithful functor $|-|:\STCOMP\to\COMPHAUS$ with fully faithful left adjoint $\COMPHAUS\hrw\STCOMP$ which is just the inclusion functor.

\begin{proposition}
For a mono-source $(f_i:X\to X_i)_{i\in I}$ in $\STCOMP$, the following assertions are equivalent.
\begin{eqcond}
\item $(f_i:X\to X_i)_{i\in I}$ is initial with respect to $|-|:\STCOMP\to\COMPHAUS$.
\item $(f_i:X\to X_i)_{i\in I}$ is initial with respect to $|-|:\STCOMP\to\SET$.
\item $(f_i:X\to X_i)_{i\in I}$ is initial with respect to $|-|:\TOP\to\SET$.
\end{eqcond}
\end{proposition}
\begin{proof}
The equivalence (i)$\RLw$(iii) follows from Theorem \ref{thm:OCHvsSTC} and Proposition \ref{prop:InitialSources}, and (i)$\RLw$(ii) follows from the fact that every monosource in $\COMPHAUS$ is initial with respect to $|-|:\COMPHAUS\to\SET$.
\end{proof}

In the sequel we call a mono-source in $\STCOMP$ initial if it is initial with respect to any of the above-mentioned forgetful functors, and similarly for mono-sources in $\ORDCH_\sep$. A stably compact space $X$ is called \df{spectral} whenever the source $(X\to 2)$ of all spectral maps from $X$ into the Sierpi\'nski space $2=\{0,1\}$ (with $\{1\}$ open) is point-separating and initial, that is, the compact open subsets of $X$ form a basis for the topology. Correspondingly, an anti-symmetric ordered compact Hausdorff space $X$ is a \df{Priestley space} (see \cite{Pri70,Pri72}) whenever the source $(X\to 2)$ of all continuous monotone maps from $X$ into $2=\{0\le 1\}$ (with the discrete topology) is point-separating and initial. We let $\SPEC$ denote the category of spectral spaces and spectral maps. For each object $X$ in $\STCOMP$, the (Epi,initial mono-source)-factorisation of $(X\to 2)$ in $\STCOMP$ provides a $\SPEC$-reflection of $X$, therefore $\SPEC$ is a reflective subcategory of $\STCOMP$. A compact Hausdorff space $X$ is a \df{Stone space} whenever the source $(X\to 2)$ of all continuous maps from $X$ into the discrete space $2=\{0,1\}$ is point-separating (and hence initial with respect to $|-|:\COMPHAUS\to\SET$), which amounts to saying that the simultaneously open and closed subsets of $X$ form a basis for the topology. Hence, a topological space $X$ is a Stone space if and only if $X$ is spectral and Hausdorff. By definition, the functor $|-|:\STCOMP\to\COMPHAUS$ restricts to $|-|:\SPEC\to\STONE$ which has the inclusion functor $\STONE\hrw\SPEC$ as a left adjoint.

The \df{lower Vietoris monad} $\mV=\vmonad$ on $\TOP$ consists of the functor $V:\TOP\to\TOP$ which sends a topological space $X$ to the space
\[
 VX=\{A\subseteq X\mid A\text{ is closed}\}
\]
with the topology generated by the sets
\[
 U^\Diamond=\{A\in VX\mid A\cap U\neq\varnothing\}\hspace{2em}(\text{$U\subseteq X$ open}),
\]
and $Vf:VX\to VY$ sends $A$ to $\overline{f[A]}$, for $f:X\to Y$ in $\TOP$; and the unit $e$ and the multiplication $m$ of $\mV$ are given by
\begin{align*}
 e_X:X\to VX,\,x\mapsto\overline{\{x\}} &&\text{and}&& m_X:VVX\to VX,\,\calA\mapsto\bigcup\calA
\end{align*}
respectively. Since $\bigcup_{i\in I}U_i^\Diamond=\left(\bigcup_{i\in I}U_i\right)^\Diamond$ it is enough to consider basic opens of $X$ in the definition of the topology of $VX$. The underlying order of $VX$ is the opposite of subset inclusion, that is, $A\le B$ if and only if $A\supseteq B$, for all $A,B\in VX$. Also note that, for all $U\subseteq X$ open and $\calA\subseteq VX$ closed,
\begin{align*}
 e_X^{-1}(U^\Diamond)&=U, & \bigcup\calA&=e_X^{-1}(\calA), & m_X^{-1}(U^\Diamond)&=(U^\Diamond)^\Diamond.
\end{align*}

\begin{lemma}\label{lem:VvsSum}
For topological spaces $X_1$ and $X_2$, $V(X_1+X_2)\simeq VX_1\times VX_2$.
\end{lemma}
\begin{proof}
We put $f_1:V(X_1+X_2)\to VX_1,\,C\mapsto C\cap X$ and  $f_2:V(X_1+X_2)\to VX_2,\,C\mapsto C\cap Y$. For open subsets $U_1\subseteq X_1$ and $U_2\subseteq X_2$,
\begin{align*}
 f_1^{-1}(U_1^\Diamond)&=\{C\in V(X_1+X_2)\mid C\cap U_1\neq\varnothing\} \intertext{and}
 f_2^{-1}(U_2^\Diamond)&=\{C\in V(X_1+X_2)\mid C\cap U_2\neq\varnothing\},
\end{align*}
hence both $f_1$ and $f_2$ are continuous. The induced continuous map $f:V(X_1+X_2)\to VX_1\times VX_2$ is bijective with inverse $g:VX_1\times VX_2\to V(X_1+X_2),\,(A_1,A_2)\mapsto A_1+A_2$. For an open subset $W\subseteq X+Y$ we put $U_1=W\cap X_1$ and $U_2=W\cap X_2$, and observe that
\[
 g^{-1}(W)=(U_1^\Diamond\times VX_2)\cup (VX_1\times U_2^\Diamond)
\]
is open in $VX_1\times VX_2$. Consequently, $g$ is continuous as well.
\end{proof}

\begin{lemma}\label{lem:VvsProd}
For topological spaces $X_1$ and $X_2$, the map \[\Pi:V(X_1)\times V(X_2)\to V(X_1\times X_2),\,(A_1,A_2)\mapsto A_1\times A_2\] is continuous and left adjoint to $\can:=\langle V\pi_1,V\pi_2\rangle: V(X_1\times X_2)\to V(X_1)\times V(X_2)$. Hence, if $V(X_1)$ and $V(X_2)$ are stably compact spaces, then $\Pi$ is spectral. 
\end{lemma}
\begin{proof}
For open subsets $U_1\subseteq X_1$ and $U_2\subseteq X_2$, 
\[
 \Pi^{-1}((U_1\times U_2)^\Diamond)=U_1^\Diamond\times U_2^\Diamond,
\]
hence $\Pi$ is continuous. Given now $(A_1,A_2)\in V(X_1)\times V(X_2)$ and $W\in V(X_1\times X_2)$, we find 
\begin{align*}
 (A,B)\le \can\cdot\Pi(A,B) &&\text{and}&& W\ge\Pi\cdot\can(W),
\end{align*}
hence $\Pi\dashv\can$. 
\end{proof}

\begin{proposition}
\begin{enumerate}
\item For every topological space $X$, $VX$ is sober.
\item A topological space $X$ is core-compact if and only if $VX$ is stably compact. Hence, if $X$ is sober, then $X$ is locally compact if and only if $VX$ is stably compact.
\item If $f:X\to Y$ is a spectral map between stably compact spaces, then $Vf:VX\to VY$ is spectral
\item If $X$ is stably compact, then $e_X:X\to VX$ and $m_X:VVX\to VX$ are spectral.
\item A stably compact space $X$ is spectral if and only if $VX$ is spectral.
\end{enumerate}
\end{proposition}
\begin{proof}
It is shown in \citep{Hof12a} that every ultrafilter on $VX$ has a smallest convergence point. Since every irreducible closed subset of a topological space is the set of limit points of some ultrafilter, this shows that $VX$ is sober. Proofs for all other properties can be also found in \citep{Hof12a}, for instance.
\end{proof}

In particular, the monad $\mV=\vmonad$ on $\TOP$ restricts to a monad on $\SPEC$, also denoted as $\mV=\vmonad$. We cannot restrict $\mV$ further to $\STONE$ since $VX$ is not even T$_1$, except for $X=\varnothing$. However, we can transfer $\mV$ to a monad $\mtV=\tvmonad$ on $\STONE$ via the adjunction $\SPEC\rightleftarrows\STONE$. Explicitly, $\hat{V}X$ is again the set of all closed subsets of $X$ but now with the topology generated by the sets
\begin{align*}
 U^\Diamond\hspace{1em}\text{($U\subseteq X$ open)} &&\text{and}&& \{A\subseteq X\text{ closed}\mid A\cap K=\varnothing\}\hspace{1em}\text{($K\subseteq X$ compact)}.
\end{align*}
The unit $e$ and the multiplication $m$ are defined as above, but note that $e_X(x)=\{x\}$ since $X$ is Hausdorff. 

\section{Kleisli dualities}\label{sect:KleisliDual}

The aim of this section is to explain how Theorem \ref{thm:KleisiEq} can be used to extend duality results to larger categories. Motivated by our main examples, we consider here natural dualities (see \citep{PT91} and \citep{CD98}) with a category of algebras on one side. 

\begin{assumption}\label{ass:General}
Throughout this section we let $\catfont{X}$ be a category equipped with a faithful functor $|-|:\catfont{X}\to \SET$. Moreover, we fix an object $\tilde{X}$ in $\catfont{X}$ and assume that $\catfont{X}$ has all powers of $\tilde{X}$ and all equalisers of pairs of morphisms between powers of $\tilde{X}$, and that $|-|$ preserves these limits. Furthermore, let $\Omega=(\Omega_0,\delta)$ be a signature, that is, $\Omega_0$ is a class (of operation symbols) and $\delta:\Omega_0\to\Card$ assigns to each operation symbol its arity. We write $\Alg{\Omega}$ to denote the category of $\Omega$-algebras and $\Omega$-homomorphisms. We also assume that an $\Omega$-algebra $\tilde{B}$ is given with $|\tilde{X}|=|\tilde{B}|$ and so that, for every operation symbol from $\Omega$ with arity $k$, the operation $|\tilde{B}|^k\to|\tilde{B}|$ underlies an $\catfont{X}$-morphism $\tilde{X}^k\to\tilde{X}$. In the sequel $\catfont{B}$ denotes the full subcategory of $\Alg{\Omega}$ defined by those $\Omega$-algebras $B$ where the source $\Alg{\Omega}(B,\tilde{B})$ is point-separating.
\end{assumption}

The following result is well-known in duality theory, see \citep[VI.4]{Joh86} or \citep[Lemma 1.5]{Hof02}, for instance.

\begin{theorem}\label{thm:BasicAdj}
Under the assumptions above, there is an adjunction
\[
 (F\dashv G,\eta,\eps):\catfont{B}^\op\rightleftarrows\catfont{X}
\]
so that
\begin{align*}
 |G|=\catfont{B}(-,\tilde{B}) &&\text{and}&& |F|=\catfont{X}(-,\tilde{X})
\end{align*}
and, for every object $B$ in $\catfont{B}$ and every object $X$ in $\catfont{X}$, the objects $GB$ and $FX$ have the $|-|$-initial structure with respect to the sources
\[
 (\catfont{B}(B,\tilde{B})\xrightarrow{\ev_x}|\tilde{B}|=|\tilde{X}|,\,h\mapsto h(x))_{x\in|B|}
\]
respectively
\[
(\catfont{X}(X,\tilde{X})\xrightarrow{\ev_x}|\tilde{X}|=|\tilde{B}|,\,h\mapsto h(x))_{x\in|X|}.
\]
Moreover, both $\eta_X$ and $\eps_B$ send $x$ to $\ev_x$.
\end{theorem}

Note that the adjunction above can be restricted to a full subcategory of $\catfont{B}^\op$ which includes the image of $F$. For instance, if $\Omega'=(\Omega_0',\delta')$ is a sub-signature of $\Omega$ (that is, $\Omega_0'\subseteq \Omega_0$ and $\delta'$ is the restriction of $\delta$ to $\Omega_0'$), applying the result above to $\Omega'$ in lieu of $\Omega$ yields an adjunction $(F'\dashv G',\eta',\eps'):\catfont{B'}^\op\rightleftarrows\catfont{X}$. We write $\catfont{B}_{\Omega'}$ to denote the category of $\Omega$-algebras and $\Omega'$-homomorphism. Then $\catfont{B}_{\Omega'}$ is equivalent to a full subcategory of $\catfont{B'}$ which includes the image of $F'$, hence we obtain an adjunction
\[
 (F'\dashv G',\eta',\eps'):\catfont{B}_{\Omega'}^\op\rightleftarrows\catfont{X}.
\]

\begin{lemma}
Assume that the natural transformation $\eps$ of Theorem \ref{thm:BasicAdj} is a natural isomorphism. Then $(F'\dashv G',\eta',\eps'):\catfont{B}_{\Omega'}^\op\rightleftarrows\catfont{X}$ is a Kleisli adjunction.
\end{lemma}

In order to identify the induced monad with a more familiar one, we consider now the following situation.

\begin{assumption}\label{ass:Monad}
Let $\mT=\monad$ be a monad on $\catfont{X}$ with $T(X_1)\simeq\tilde{X}$, for some object $X_1$ in $\catfont{X}$; to simplify notation, we assume that $\tilde{X}$ is chosen so that $T(X_1)=\tilde{X}$. Hence, $\tilde{X}$ is a $\mT$-algebra, and therefore every $h:X\to\tilde{X}$ has a unique extension $\bar{h}:TX\to\tilde{X}$ where $\bar{h}$ is a $\mT$-homomorphism with $\bar{h}\cdot e_X=h$. For every object $X$ in $\catfont{X}$ we assume:
\begin{itemize}
\item the map
\[
 \bar{(-)}:\catfont{X}(X,\tilde{X})\to\catfont{X}(TX,\tilde{X}),\, h\mapsto \bar{h}
\]
preserves all operations from $\Omega'$, that is, $\bar{(-)}$ is a morphism in $\catfont{B}_{\Omega'}^\op$ of type $F'X\to F'TX$, and
\item the source $(\bar{h}:TX\to\tilde{X})_{h\in\catfont{X}(X,\tilde{X})}$ is point-separating and $|-|$-initial.
\end{itemize}
\end{assumption}

\begin{proposition}\label{prop:J}
Under the assumptions above, the hom-functor $\catfont{X}_\mT(-,X_1)$ lifts to a functor $J:\catfont{X}_\mT\to\catfont{B}_{\Omega'}^\op$ so that the diagram
\[
 \xymatrix{\catfont{X}_\mT\ar[rr]^J && \catfont{B}_{\Omega'}^\op\\
 & \catfont{X}\ar[ul]^{F_\mT}\ar[ur]_{F'}}
\]
commutes. The induced monad morphism $j$ is component-wise an embedding with $X$-component
\[
 j_X:TX\to G'F'X,\,\fx\mapsto(h\mapsto\bar{h}(\fx)),
\]
for each object $X$ in $\catfont{X}$.
\end{proposition}
\begin{proof}
For each object $X$ in $\catfont{X}_\mT$, one has $\catfont{X}_\mT(X,X_1)=\catfont{X}(X,\tilde{X})$, and we put $JX=F'X$. Given a morphism $r:X\kto Y$ in $\catfont{X}_\mT$, the map $\catfont{X}_\mT(r,X_1):\catfont{X}_\mT(Y,X_1)\to \catfont{X}_\mT(X,X_1)$ sends $h\in \catfont{X}_\mT(Y,X_1)=\catfont{X}(Y,\tilde{X})$ to $m_{X_1}\cdot Th\cdot r$. Hence, this map can be written as the composite
\[
 \catfont{X}(Y,\tilde{X})\xrightarrow{\bar{(-)}}\catfont{X}(TY,\tilde{X})\xrightarrow{\catfont{X}(r,\tilde{X})}\catfont{X}(X,\tilde{X})
\]
and therefore underlies a morphism $Jr:JY\to JX$ in $\catfont{B}_{\Omega'}$. In conclusion, this construction yields a functor $J:\catfont{X}_\mT\to\catfont{B}_{\Omega'}^\op$ which clearly satisfies $J F_\mT=F'$. To describe the corresponding natural transformation $\iota:G_\mT\to G'J$, we note that $\eta'_{G_\mT X}$ is given by the map
\[
 \eta'_{G_\mT X}:TX\to\catfont{B}_{\Omega'}(\catfont{X}(TX,\tilde{X}),\tilde{B}),\,\fx\mapsto\ev_\fx,
\]
$\eps_X:TX\kto X$ of $F_\mT\dashv G_\mT$ corresponds to the $\catfont{X}$-morphism $1_TX:TX\to TX$ and therefore $J\eps_X$ sends $h\in\catfont{X}(X,\tilde{X})$ to $\bar{h}$, and $G' J\eps_X$ sends $\Phi:JTX\to\tilde{B}$ to the map $\catfont{X}(X,\tilde{X})\to\tilde{B},\,h\mapsto\Phi(\bar{h})$. Putting all together, $\iota_X:G_\mT X\to G'JX$ sends $\fx\in TX$ to the map $\catfont{X}(X,\tilde{X})\to\tilde{B},\,h\mapsto\bar{h}(\fx)$; and from $j_X=\iota_{F_\mT X}$ we obtain the desired description of $j$. Finally, for every object $X$ in $\catfont{X}$ and every morphism $h:X\to\tilde{X}$ in $\catfont{X}$, the diagram
\[
 \xymatrix{TX\ar[r]^-{j_X}\ar[dr]_{\bar{h}} & \catfont{B}_{\Omega'}(F'X,\tilde{B})\ar[d]^{\ev_h}\\ & X}
\]
commutes and therefore $j_X:TX\to G'F'X$ is an embedding in $\catfont{X}$. 
\end{proof}

By construction, the diagram
\[
 \xymatrix{\catfont{X}_\mT\ar[r]^J & \catfont{B}_{\Omega'}^\op\\ \catfont{X}\ar[u]\ar[r]_F & \catfont{B}^\op\ar[u]}
\]
of functors commutes. If both $J$ and $F$ are equivalence functors, then a category constructed from $\catfont{X}\to\catfont{X}_\mT$ is dually equivalent to the category obtained from $\catfont{B}\to\catfont{B}_{\Omega'}$ by the same construction. For instance, the category $\Coalg{T:\catfont{X}\to\catfont{X}}$ of coalgebras for $T$ has as objects $\catfont{X}$-morphisms $e:X\to TX$, and a morphism $f:(X,e)\to (X',e')$ in $\Coalg{T:\catfont{X}\to\catfont{X}}$ is a $\catfont{X}$-morphism $f:X\to X'$ with $Tf\cdot e=e'\cdot f$. Equivalently, we can think of the objects of $\Coalg{T:\catfont{X}\to\catfont{X}}$ as endomorphisms $e:X\kto X$ in $\catfont{X}_\mT$, and $Tf\cdot e=e'\cdot f$ means precisely that the diagram
\[
 \xymatrix{X\ar@{-^>}[r]^{F_\mT f} & Y\\ X\ar@{-^>}[u]^e\ar@{-^>}[r]_{F_\mT f} & Y\ar@{-^>}[u]_{e'}}
\]
commutes in $\catfont{X}_\mT$. Therefore:

\begin{corollary}\label{cor:JwithOp}
Assume that $J:\catfont{X}_\mT\to\catfont{B}_{\Omega'}^\op$ and $F:\catfont{X}\to\catfont{B}^\op$ are equivalence functors. then the category $\Coalg{T:\catfont{X}\to\catfont{X}}$ is dually equivalent to the category $\CatOp{\catfont{B}}{\Omega'}$ of ``$\catfont{B}$-objects with operator'', that is, the objects of $\CatOp{\catfont{B}}{\Omega'}$ are $\catfont{B}$-objects $B$ equipped with an unitary operation $h:B\to B$ which preserves all operations from $\Omega'$, and a morphism $f:(B,h)\to (B',h')$ in $\CatOp{\catfont{B}}{\Omega'}$ is a $\catfont{B}$-morphism $f:B\to B'$ with $h'\cdot f=f\cdot h$.
\end{corollary}

\begin{examples}
\begin{trivlist}
\item[\hskip\labelsep (1)] We consider first $\catfont{X}=\SET$ with $\tilde{X}=2$, $\catfont{B}$ is the category $\CABOOL$ of complete atomic Boolean algebras and functions preserving all suprema and all infima with $\tilde{B}=2$ being the two-element Boolean algebra. It is well-known that the corresponding adjunction
\[
 (F\dashv G,\eta,\eps):\CABOOL^\op\rightleftarrows\SET
\]
is actually an equivalence (see \citep{Joh86}, for instance). We consider now the category $\CABOOL_{\bigvee}$ of complete atomic Boolean algebras and functions preserving all suprema as well as the powerset monad $\mP=\pmonad$ on $\SET$. Clearly, $P1=2$, and, for every map $h:X\to 2$, the extension $\bar{h}:PX\to 2$ is defined by
\[
 \bar{h}(A)=1\iff \exists x\in A\,.\,h(x)=1,
\]
for all $A\subseteq X$. In other words, $\bar{h}(A)=\bigvee\{h(x)\mid x\in A\}$, and we conclude that the map $\bar{(-)}$ of Assumption \ref{ass:Monad} preserves suprema. A quick calculation shows that the monad morphism $j$ induced by $J$ is given by
\[
 j_X:PX\to\CABOOL_{\bigvee}(F'X,2),\,
A\mapsto \Phi_A:F'X\to 2,\,h\mapsto
\begin{cases}
 0 & \text{if $A_h\subseteq A^\complement$,}\\
 1 & \text{else,}
\end{cases}
\]
where $A_h=h^{-1}(1)$. Since every suprema-preserving map $\Phi:F'X\to 2$ is completely determined by the largest element $h\in F'X$ with $\Phi(h)=0$, each component of $j$ is a bijection, that is, $j$ is an isomorphism. Therefore, from Proposition \ref{prop:J} and Corollary \ref{cor:JwithOp} we obtain:
\begin{align*}
\REL\simeq\CABOOL_{\bigvee}^\op &&\text{and}&& \Coalg{P:\SET\to\SET}\simeq \CatOp{\CABOOL}{\bigvee}^\op.
\end{align*}
Here $\CatOp{\CABOOL}{\bigvee}$ denotes the category with objects complete atomic Boolean algebras $B$ equipped with a unary operation $h:B\to B$ which preserves suprema, and a morphism $h:(B,h)\to B'h')$ in $\CatOp{\CABOOL}{\bigvee}$ is a morphism $h:B\to B'$ in $\CABOOL$ with $f\cdot h=h'\cdot f$.

\item[\hskip\labelsep (2)] Similarly, if we consider the category $\CABOOL_{\top,\wedge}$ of complete atomic Boolean algebras and finite infima preserving maps instead, the adjunction
\[
 (F'\dashv G',\eta',\eps'):\CABOOL_{\top,\wedge}^\op\rightleftarrows\SET
\]
induces a monad isomorphic to the filter monad $\mF$ on $\SET$. Hence:
\begin{align*}
 \SET_\mF\simeq\CABOOL_{\top,\wedge}^\op &&\text{and}&& \Coalg{F:\SET\to\SET}\simeq\CatOp{\CABOOL}{\top,\wedge}^\op.
\end{align*}

\item[\hskip\labelsep (3)] Our next example involves $\catfont{X}=\TOP$ and $\tilde{X}=2=\{0,1\}$ the Sierpi\'nski space with $\{1\}$ open; $\catfont{B}$ is the category $\SFRM$ of spatial frames and frame-homomorphisms and  $\tilde{B}=2$ is the two-element frame. We obtain the adjunction
\[
 (F\dashv G,\eta,\eps):\SFRM^\op\rightleftarrows\TOP
\]
where $\eps_B$ is an isomorphism for every spatial frame $B$, and $\eta_X$ is an isomorphism if and only if $X$ is sober, for every topological space $X$. Similar to the previous example, we consider the category $\SFRM_{\top,\wedge}$ of spatial frames and finite infima preserving maps, and obtain a Kleisli adjunction
\[
 (F'\dashv G',\eta',\eps'):\SFRM_{\top,\wedge}^\op\rightleftarrows\TOP.
\]
It is well-known that the induced monad is isomorphic to the filter monad $\mF=\ftmonad$ on $\TOP$. In fact, $F_\tau 1$ is the Sierpi\'nski space, and
\[
 \bar{(-)}:\TOP(X,2)\to\TOP(F_\tau X,2)
\]
sends $h:X\to 2$ the characteristic map of $(U_h)^\#$ (where $U_h=h^{-1}(1)$), hence $\bar{(-)}$ preserves finite infima. By definition, the topology on $F_\tau X$ is generated by all sets $(U_h)^\#$ for $h:X\to 2$ continuous, which tells us that the point-separating source $(\bar{h}:F_\tau X\to 2)_{h\in\TOP(X,2)}$ is initial. The $X$-component of the monad morphism $j$ is given by
\[
 j_X:F_\tau X\to\SFRM_{\top,\wedge}(F'X,2),\ff\mapsto \Phi_\ff:F'X\to 2,\,h\mapsto
\begin{cases}
 1 & \text{$U_h\in\ff$,}\\
 0 & \text{else,}
\end{cases}
\]
hence $j_X$ is surjective and therefore an isomorphism. In conclusion:
\begin{align*}
 \TOP_\mF\simeq\SFRM_{\top,\wedge} &&\text{and}&& \Coalg{F:\SOB\to\SOB}\simeq\CatOp{\SFRM}{\top,\wedge}^\op.
\end{align*}

\item[\hskip\labelsep (4)]
We substitute now in the example above $\SFRM_{\top,\wedge}$ with the category $\SFRM_{\bigvee}$ of spatial frames and suprema preserving maps. We show that the monad induced by the adjunction
\[
 (F'\dashv G',\eta',\eps'):\SFRM_{\bigvee}^\op\rightleftarrows\TOP.
\]
is isomorphic to the lower Vietoris monad $\mV=\vmonad$ on $\TOP$ (see Section \ref{sect:StablyComp}). Clearly, $V1$ is the Sierpi\'nski space, and the map
\[
 \bar{(-)}:\TOP(X,2)\to\TOP(VX,2)
\]
sends $h:X\to 2$ to the characteristic map of $(U_h)^\Diamond$. Therefore $\bar{(-)}$ preserves all suprema and $VX$ has by definition the initial topology with respect to the point-separating source $(\bar{h}:VX\to 2)_{h\in\TOP(X,2)}$. Similar to the first example, the monad morphism $j$ induced by $J$ is given by
\[
 j_X:VX\to\SFRM_{\bigvee}(F'X,2),\,
A\mapsto \Phi_A:F'X\to 2,\,h\mapsto
\begin{cases}
 0 & \text{if $U_h\subseteq A^\complement$,}\\
 1 & \text{else,}
\end{cases}
\]
hence $j$ is an isomorphism and we obtain:
\begin{align*}
\TOP_\mV\simeq\SFRM_{\bigvee}^\op &&\text{and}&& \Coalg{V:\SOB\to\SOB}\simeq\CatOp{\SFRM}{\bigvee}^\op.
\end{align*}

Clearly, a morphism $X\kto Y$ in $\TOP_\mV$ corresponds to a relation $X\relto Y$, and composition in $\TOP_\mV$ corresponds to relational composition. We call a relation $r:X\relto Y$ between topological spaces \df{continuous relation} whenever the corresponding map $\mate{r}:X\to PY$ factors as $X\xrightarrow{\mate{r}}VY\hrw PY$ and, moreover, $\mate{r}:X\to VY$ is continuous. Writing $\TOPREL$ for the category of topological spaces and continuous relations, we have $\TOPREL\simeq\SFRM_{\bigvee}^\op$. The functor $\TOP\to\TOP_\mV$ corresponds to the functor $\TOP\to\TOPREL$ which sends $f:X\to Y$ to
\[
f_*:X\relto Y,\,x\,f_*\,y\iff f(x)\le y,
\]
where $\le$ refers to the underlying order of $Y$. We also note that every continuous relation $r:X\relto Y$ satisfies
\[
 (x\le x'\,r\,y'\le y)\;\Rw\; x\,r\,y,
\]
for all $x,x'\in X$ and $y,y'\in Y$.

\item[\hskip\labelsep (5)] We consider now the category $\catfont{X}=\SPEC$ of spectral spaces and spectral maps. As shown in \citep{Sto38}, $\SPEC$ is equivalent to the dual of the category $\catfont{B}=\DLAT$ of distributive lattices (with top and bottom element) and lattice homomorphisms, and this equivalence is induced by $\tilde{X}=2$ the Sierpi\'nski space and $\tilde{B}=2$ the two-element lattice. With $\DLAT_{\bot,\vee}$ denoting the category of distributive lattices and finite suprema preserving maps, we obtain a Kleisli adjunction
\[
 (F'\dashv G',\eta',\eps'):\DLAT_{\bot,\vee}^\op\rightleftarrows\SPEC.
\]
For the lower Vietoris monad $\mV=\vmonad$ on $\SPEC$, $V1$ is the Sierpi\'nski space and the map $\bar{(-)}:\SPEC(X,2)\to\SPEC(VX,2)$ is the restriction of the corresponding map of Example 4 above and therefore preserves finite suprema; $(\bar{h}:VX\to 2)_{h\in\SPEC(X,2)}$ is point-separating and $VX$ has the initial topology ($=$ initial $\SPEC$-structure). As above, the monad morphism $j$ induced by $J$ is given by
\[
 j_X:VX\to\DLAT_{\bot,\vee}(F'X,2),\,
A\mapsto \Phi_A:F'X\to 2,\,h\mapsto
\begin{cases}
 0 & \text{if $U_h\subseteq A^\complement$,}\\
 1 & \text{else,}
\end{cases}
\]
and compactness of the sets $U_h$ (for $h:X\to 2$ in $\SPEC$) assures that $j_X$ is surjective. Therefore:
\begin{align*}
 \SPEC_\mV\simeq\DLAT_{\bot,\vee}^\op &&\text{and}&& \Coalg{V:\SPEC\to\SPEC}\simeq\CatOp{\DLAT}{\bot,\vee}^\op.
\end{align*}
The above mentioned dualities were also obtained in \citep{CLP91}, \citep{Pet96} and \citep{BKR07}.
Put differently, $\SPECREL\simeq\DLAT_{\bot,\vee}^\op$ where $\SPECREL$ denotes the category of spectral spaces and spectral relations (that is, relations $r:X\relto Y$ between spectral spaces so that $\mate{r}:X\to PY$ factors as $X\xrightarrow{\mate{r}}VY\hrw PY$ and $\mate{r}:X\to VY$ is spectral). 

\item[\hskip\labelsep (6)] Finally, we describe the dualities involving Boolean algebras mentioned in the Introduction. By \citep{Sto36}, $\catfont{X}=\STONE$ is equivalent to the dual of the category $\catfont{B}=\BOOL$ of Boolean algebras and homomorphisms. As above, for the monad $\mtV=\tvmonad$ on $\STONE$ one obtains
\begin{align*}
 \STONEREL\simeq\STONE_{\mtV}\simeq\BOOL_{\bot,\vee}^\op &&\text{and}&& \Coalg{\hat{V}:\STONE\to\STONE}\simeq\CatOp{\BOOL}{\bot,\vee}^\op.
\end{align*}
Here $\STONEREL$ denotes full subcategory of $\SPECREL$ defined by all Stone spaces.
\end{trivlist}
\end{examples}

\section{Monoidal structures}

The starting point of this section is the observation that the topological product of two spectral spaces is \emph{not} their product in $\SPECREL$, as it can be seen already in the simplest case:
\[
 \SPECREL(1,1\times 1)\not\simeq\SPECREL(1,1)\times\SPECREL(1,1).
\]
The formula above suggests that the product $1\times 1$ in $\SPECREL$ should be the two-element discrete space. In fact, we show now that products in $\SPECREL$ are given by coproducts in $\SPEC$. We will need a piece of notation. For a morphism $f:X\to Y$ in $\SPEC$, we define a relation $f^*:Y\relto X$ as $y\,f^*\,x$ whenever $y\le f(x)$, for all $x\in X$ and $y\in Y$. We emphasize that $f^*$ does not need to be a spectral relation; however, if $f$ is an open map, then $f^*$ is spectral (see \citep[Proposition 8.4]{Hof12a}). Also note that, for a morphism $r:Z\relto Y$ in $\SPECREL$ and $z\in Z$, $x\in X$,
\[
 z\,(f^*\cdot r)\,x\iff\exists y\in Y\,.\,(z\,r\,y)\wedge (y\le f(x))\iff z\,r\,f(x).
\]

\begin{lemma}
For spectral spaces $X_1$ and $X_2$, the product of $X_1$ and $X_2$ in $\SPECREL$ is given by their topological sum (which is also their coproduct in $\SPEC$ and in $\SPECREL$). If $i_1:X_1\to X_1+X_2$ and $i_2:X_2\to X_1+X_2$ denote the coproduct injections, then $i_1^*:X_1+X_2\relto X_1$ and $i_2^*:X_1+X_2\relto X_2$ are the two product projections.
\end{lemma}
\begin{proof}
Given morphisms $r:Z\relto X_1$ and $s:Z\relto X_2$ in $\SPECREL$, let
\[
 \langle r,s\rangle:Z\relto X_1+X_2
\]
be the spectral relation which corresponds to the composite
\[
 Z\xrightarrow{\,\langle\mate{r},\mate{s}\rangle} V(X_1)\times V(X_2)\simeq V(X_1+X_2)
\]
in $\SPEC$ (see Lemma \ref{lem:VvsSum}). Element-wise, for $z\in Z$ and $x\in X_1+X_2$, one has
\begin{equation}\label{eq:CanRelProd}
 z\,\langle r,s\rangle\,x\iff (x\in X_1\wedge z\,r\,x)\vee(x\in X_2\wedge z\,s\,x);
\end{equation}
and this gives at once $i_1^*\cdot\langle r,s\rangle=r$ and $i_2^*\cdot\langle r,s\rangle=s$. Finally, a spectral relation $t:Z\relto X_1+X_2$ with $i_1^*\cdot t=r$ and $i_2^*\cdot t=s$ is necessarily described by the right hand side of \eqref{eq:CanRelProd}, hence $t=\langle r,s\rangle$.
\end{proof}

A similar result holds for $\STONEREL$, and the empty space is an initial and a terminal object in both $\SPECREL$ and $\STONEREL$. 

\begin{proposition}
The categories $\SPECREL$, $\STONEREL$, $\DLAT_{\bot,\vee}$ and $\BOOL_{\bot,\vee}$ have finite products and finite coproducts which coincide. 
\end{proposition}

Note that finite (co)products in the categories $\DLAT_{\bot,\vee}$ and $\BOOL_{\bot,\vee}$ are given by products in $\DLAT$ and $\BOOL$ respectively. For objects $X$ and $Y$ in $\SPECREL$, we denote their topological product as $X\otimes Y$. For spectral relations $r:X\relto X'$ and $s:Y\relto Y'$, we define a relation $r\otimes s:X\otimes Y\relto X'\otimes Y'$ where $(x,y)\,(r\otimes s)\,(x',y')$ whenever $x\,r\,x'$ and $r\,s\,y'$. Then $r\otimes s$ is indeed spectral since, with $\mate{r}:X\to V(X')$ and $\mate{s}:Y\to V(Y')$ denoting the corresponding spectral maps, $r\otimes s$ corresponds to the map
\[
 X\times Y\xrightarrow{\mate{r}\times \mate{s}}V(X')\times V(Y')\xrightarrow{\,\Pi\,} V(X'\times Y')
\]
in $\SPEC$ (see Lemma \ref{lem:VvsProd}). One easily verifies that this construction defines a functor
\[
 \otimes:\SPECREL\times\SPECREL\to\SPECREL.
\]
We shall see that on the algebraic side this operation corresponds to a tensor product which represents bimorphisms.

For objects $L,M,N$ in $\DLAT_{\bot,\vee}$, we call a map $f:L\times M\to N$ a \df{bimorphism} whenever, for all $x\in L$, $f(x,-):M\to N$ preserves finite suprema and, for all $y\in M$, $f(-,y):L\to N$ preserves finite suprema. Note that a bimorphism is necessarily monotone. We say that an object $P$ in $\DLAT_{\bot,\vee}$ \df{represents bimorphisms} $f:L\times M\to N$ whenever there is a bimorphism $p:L\times M\to P$ such that, for every bimorphism $f:L\times M\to N$, there exists a unique morphism $\bar{f}:P\to N$ with $\bar{f}\cdot p=f$. If it exists, $P$ is unique up to isomorphism, and we denote such an object as $L\otimes M$ and refer to $p:L\times M\to L\otimes M$ as the \df{tensor product} of $L$ and $M$.

\begin{lemma}
For spectral spaces $X$ and $Y$, the tensor product of $JX$ and $JY$ exists and is given by
\[
 p:JX\times JY\to J(X\otimes Y),\,(A,B)\mapsto A\times B.
\]
\end{lemma}
\begin{proof}
For each spectral space $X$, we identify the elements of $JX$ with the compact open subsets of $X$. Let $X$ and $Y$ be spectral spaces. Clearly, $p:JX\times JY\to J(X\otimes Y)$ is a bimorphism. Let $f:SX\times SY\to N$ be a bimorphism. It is enough to consider the case $N=SZ$, for a spectral space $Z$. We define a map
\[
 g:J(X\otimes Y)\to PZ,\,W\mapsto\bigcup\{f(A,B)\mid A\in JX, B\in JY,A\times B\subseteq W\}.
\]
Then $g$ is monotone, $g(\varnothing)=\varnothing$ and, for all $A\in JX$ and $B\in JY$, $g(A\times B)=f(A,B)$. We show now that $g$ preserves binary suprema; since $g$ is monotone, it is enough to show $g(W\cup W')\subseteq g(W)\cup g(W')$. To this end, fix $A\in JX$ and $B\in JY$ with $A\times B\subseteq W\cup W'$. Since $W$ and $W'$ are compact opens, there are finite sets $I$ and $I'$ and compact opens $A_i,A_j'\subseteq X$ and $B_i,B_j'\subseteq Y$ ($i\in I$, $j\in I'$) with
\begin{align}\label{eq:WFiniteUnion}
 W=\bigcup_{i\in I}A_i\times B_i &&\text{and}&& W'=\bigcup_{j\in I'}A_j'\times B_j'.
\end{align}
For $x\in A$, put
\begin{align*}
 A_x &=A\cap\bigcap\{A_i\mid i\in I,x\in A_i\}\cap\bigcap\{A'_j\mid j\in I',x\in A_j'\}
\intertext{and, for each $y\in B$, put}
 B_y &=B\cap\bigcap\{B_i\mid i\in I,y\in B_i\}\cap\bigcap\{B'_j\mid j\in I',y\in B_j'\}.
\end{align*}
Then $x\in A_x$, $y\in B_y$ and $A_x$ and $B_y$ are open and compact. Hence,
\begin{align*}
 A=\bigcup_{x\in A}A_x=A_{x_1}\cup\ldots\cup A_{x_n} &&\text{and}&&
 B=\bigcup_{y\in B}B_y=B_{y_1}\cup\ldots\cup B_{y_m},
\end{align*}
for finitely many elements $x_1,\ldots,x_n\in A$ and $y_1,\ldots,y_m\in B$. Since $f$ is a bimorphism, we obtain
\[
 f(A,B)=\bigcup\{f(A_{x_i},B_{y_j})\mid 1\le i\le n,1\le j\le m\}.
\]
Let now $z\in f(A,B)$. Then $z\in f(A_{x_k},B_{y_l})$, for some $1\le k\le n$ and $1\le l\le m$; and $(x_k,y_l)\in A\times B\subseteq W\cup W'$. Without loss of generality we assume $(x_k,y_l)\in W$. We have $x_k\in A_i$ and $y_l\in B_i$ for some $i\in I$, hence $A_{x_k}\subseteq A_i$ and $B_{y_l}\subseteq B_i$ and therefore 
\[
 z\in f(A_{x_k},B_{y_l})\subseteq f(A_i,B_i)\subseteq g(W).
\]
This proves $g(W\cup W')\subseteq g(W)\cup g(W')$. From preservation of finite suprema it also follows that $g$ takes values in $JZ$ since (using \eqref{eq:WFiniteUnion}) 
\[g(W)=\bigcup_{i\in I}g(A_i\times B_i)=\bigcup_{i\in I}f(A_i,B_i)\]
is compact (and open). This also shows that $g$ is the only finite suprema preserving map $g:J(X\otimes Y)\to PZ$ with $g\cdot p=f$, and the corestriction $\bar{f}:J(X\otimes Y)\to JZ$ of $g$ to $JZ$ is the required unique map with $\bar{f}\cdot p=f$.
\end{proof}

\begin{theorem}
For all distributive lattices $L$ and $M$, their tensor product in $\DLAT_{\bot,\vee}$ exists. Moreover, the functor $J:\SPECREL\to \DLAT_{\bot,\vee}$ satisfies $J(X\otimes Y)\simeq JX\otimes JY$, for all spectral spaces $X$ and $Y$. 
\end{theorem}

All what was said above can be restricted to Hausdorff spaces and Boolean algebras.

\begin{theorem}
For all Boolean algebras $L$ and $M$, their tensor product in $\BOOL_{\bot,\vee}$ exists and the functor $J:\STONEREL\to \BOOL_{\bot,\vee}$ satisfies $J(X\otimes Y)\simeq JX\otimes JY$, for all Stone spaces $X$ and $Y$. 
\end{theorem}

Hence, a spectral relation $X\relto X\otimes X$ corresponds to a morphism $B\otimes B\to B$ in $\DLAT_{\bot,\vee}$, which can be also seen as a map $B\times B\to B$ preserving finite infima in each variable (the corresponding result for Boolean algebras can be found in \citep{Cel03}). In fact, the (monoidal) structure on $\SPECREL$ and $\STONEREL$ can be used to express properties of relations, and preservation properties of $J$ allow to translate these properties systematically into algebraic properties of the corresponding morphism in $\DLAT_{\bot,\vee}$ respectively $\BOOL_{\bot,\vee}$. For instance, the unique map $!:X\to 1$ of a Stone spaces $X$ corresponds to the unique map $2\to A$ in $\BOOL$ sending $0$ to $\bot$ and $1$ to $\top$, and the diagonal map $\Delta:X\to X\times X$ corresponds to the map $A\otimes A\to A$ induced by the bimorphism $\wedge:A\times A\to A$. A morphism $r:X\relto Y$ in $\STONEREL$ is a total relation whenever
\begin{align*}
 \forall x\in X\exists y\in Y\,.\,x\,r\,y, &&\text{which is equivalent to}&&
 \xymatrix{X\ar[r]|-{\object@{|}}^r\ar[dr]|-{\object@{|}}_{!_*} & Y\ar[d]|-{\object@{|}}^{!_*} \\ & 1}
\end{align*}
commutes in $\STONEREL$. Hence, $r$ is total if and only if the corresponding morphism $f:B\to A$ in $\BOOL_{\bot,\vee}$ makes the diagram
\begin{align*}
 \xymatrix{B\ar[r]^f & A\\ 2\ar[u]\ar[ur]}
\end{align*}
commute, that is, $f$ preserves the top-element. Slightly more general, given two morphisms $r,s:X\relto Y$ in $\STONEREL$, we can express the property that each $x\in X$ has a ``successor'' for at least one of these relations (i.e.\ $\forall x\in X\exists y\in Y\,.\,(x\,r\,y)\vee(x\,s\,y)$) by requiring that
\[
 \xymatrix{X\ar[r]|-{\object@{|}}^-{\langle r,s\rangle} & Y+Y\ar[r]|-{\object@{|}}^-{!_*} & 1}
\]
is equal to $!_*:X\relto 1$; and this means precisely that the corresponding morphisms $f,g:B\to A$ in $\BOOL_{\bot,\vee}$ must make the diagram
\[
 \xymatrix{2\ar[r]\ar[dr] & B\times B\ar[d]^{[f,g]}\\ & A}
\]
commute, that is, $f(\top)\vee g(\top)=\top$. Similarly, $r:X\relto Y$ in $\STONEREL$ is a partial map if and only if the diagram
\[
 \xymatrix{X\otimes X\ar[r]|-{\object@{|}}^{r\otimes r} & Y\otimes Y\\
 X\ar[u]|-{\object@{|}}^{\Delta_*}\ar[r]|-{\object@{|}}_r & Y\ar[u]|-{\object@{|}}_{\Delta_*}}
\]
commutes in $\STONEREL$, hence partial maps correspond to binary infima preserving morphisms $f:A\to B$ in $\BOOL_{\bot,\vee}$. Finally, a morphism $r:X\relto Y$ in $\SPECREL$ is a partial map whenever, for each $x\in X$, the closed set $\{y\in Y\mid x\,r\,y\}$ is either empty or has a smallest element with respect to the underlying order of $Y$.
\begin{lemma}
Let $r:X\relto Y$ be in $\SPECREL$. Then the following assertions are equivalent.
\begin{eqcond}
\item The relation $r$ is a partial map.
\item For each $x\in X$, the closed set $\{y\in Y\mid x\,r\,y\}$ is either empty or down-directed.
\item The diagram
\[
 \xymatrix{X\otimes X\ar[r]|-{\object@{|}}^{r\otimes r} & Y\otimes Y\\
 X\ar[u]|-{\object@{|}}^{\Delta_*}\ar[r]|-{\object@{|}}_r & Y\ar[u]|-{\object@{|}}_{\Delta_*}}
\]
commutes in $\SPECREL$.
\end{eqcond}
\end{lemma}
\begin{proof}
Clearly (i) implies (ii), and (iii) is just a reformulation of (ii). To see (ii)$\Rw$(i), let $x\in X$ with $\{y\in Y\mid x\,r\,y\}$ down-directed. Then the closed set $\{y\in Y\mid x\,r\,y\}$ is irreducible and, since $Y$ is sober, is of the form $\overline{\{y_0\}}=\{y\in Y\mid y_0\le y\}$ for some $y_0\in Y$.
\end{proof}

We conclude that partial maps $r:X\relto Y$ in $\SPECREL$ correspond to binary infima preserving morphisms $f:A\to B$ in $\DLAT_{\bot,\vee}$.

\def\cprime{$'$}

\end{document}